\documentclass[leqno,a4paper]{article}
\usepackage{amsmath,amsthm,amssymb}
\usepackage[utf8]{inputenc}
\usepackage[T1]{fontenc}

  \usepackage{hyperref}
 \usepackage[svgnames,x11names]{xcolor}

 \hypersetup{
   colorlinks, linkcolor={Chocolate4},
   citecolor={Chocolate4}, urlcolor={DarkOliveGreen4}
  }

\usepackage{enumerate}
\usepackage{bbm}
\usepackage{todonotes}
\usepackage{mathrsfs}
\usepackage{mathabx}
\usepackage{nicefrac}
\usepackage{tikz}
\usetikzlibrary{matrix}

\usepackage{url}

\usepackage[sort,numbers]{natbib}

\providecommand{\noopsort}[1]{} 

  \usepackage{srcltx}
  \usepackage{graphicx}

  \usepackage{color}

	\def\CC{{\ifmmode{\mathbbm{C}}\else{$\mathbbm{C}$}\fi}}
	\def\EE{{\ifmmode{\mathbbm{E}}\else{$\mathbbm{E}$}\fi}}
	\def\FF{{\ifmmode{\mathbbm{F}}\else{$\mathbbm{F}$}\fi}}
	\def\HH{{\ifmmode{\mathbbm{H}}\else{$\mathbbm{H}$}\fi}}
	\def\KK{{\ifmmode{\mathbbm{K}}\else{$\mathbbm{K}$}\fi}}
	\def\NN{{\ifmmode{\mathbbm{N}}\else{$\mathbbm{N}$}\fi}}
	\def\QQ{{\ifmmode{\mathbbm{Q}}\else{$\mathbbm{Q}$}\fi}}
	\def\RR{{\ifmmode{\mathbbm{R}}\else{$\mathbbm{R}$}\fi}}
	\def\TT{{\ifmmode{\mathbbm{T}}\else{$\mathbbm{T}$}\fi}}
	\def\UU{{\ifmmode{\mathbbm{U}}\else{$\mathbbm{U}$}\fi}}
	\def\ZZ{{\ifmmode{\mathbbm{Z}}\else{$\mathbbm{Z}$}\fi}}

	\newcommand{\tend}[2]{\xrightarrow[#1\to#2]{}}
	\newcommand{\equival}[2]{\mathop{\sim}_{#1\to#2}}
	\newcommand{\Id}{\mathop{\mbox{Id}}}

	\newcommand{\ind}[1]{\mathbbmss{1}_{#1}}

	\newcommand{\Yg}[1]{Y^{(#1)}}
	\newcommand{\ph}[1]{\phi_{#1}}
    \newcommand{\bfu}{\boldsymbol{u}}
    \newcommand{\MRT}{\rm MRT}

\newtheorem{Th}{Theorem}[section]
\newtheorem{Prop}[Th]{Proposition}
\newtheorem{Lemma}[Th]{Lemma}

\theoremstyle{definition}
\newtheorem{Remark}[Th]{Remark}
\newtheorem{Def}{Definition}[section]
\newtheorem{Cor}[Th]{Corollary}

\newcommand{\beq}{\begin{equation}}
\newcommand{\eeq}{\end{equation}}

\def\scalar(#1,#2){(#1\mid#2)}

\newcommand{\xbm}{(X,{\cal B},\mu)}

\newcommand{\ycn}{(Y,{\cal C},\nu)}
\newcommand{\ot}{\otimes}
\newcommand{\ov}{\overline}

\newcommand{\bs}{\mathbb{S}}

\newcommand{\R}{{\mathbb{R}}}
\newcommand{\T}{{\mathbb{T}}}
\newcommand{\C}{{\mathbb{C}}}
\newcommand{\Z}{{\mathbb{Z}}}
\newcommand{\N}{{\mathbb{N}}}
\newcommand{\D}{{\mathbb{D}}}

\newcommand{\vep}{\varepsilon}

\newcommand{\mob}{\boldsymbol{\mu}}
\newcommand{\lio}{\boldsymbol{\lambda}}

\newcommand{\Leb}{{\rm Leb\,}}

\begin{document}
\title{On Furstenberg systems of aperiodic multiplicative functions\\ of Matom\"aki, Radziwi\l\l \ and Tao}
\author{Aleksander Gomilko \and Mariusz Lema\'nczyk \and Thierry de la Rue}

\maketitle

\begin{abstract} It is shown that in a class of counterexamples to Elliott's conjecture by Matom\"aki, Radziwi\l\l \ and Tao \cite{Ma-Ra-Ta}, the Chowla conjecture holds along a subsequence.\end{abstract}

\section{Introduction} The celebrated Chowla conjecture \cite{Ch} from 1965 predicts that for the arithmetic Liouville function $\lio$, we have
\beq\label{chowla1}
\lim_{N\to\infty}\frac1N\sum_{1\le n\le N}\lio(n+a_1)\cdot\ldots\cdot
\lio(n+a_k)=0\eeq
for any choice of $0\leq a_1<\ldots<a_k$, $k\geq1$.
As noticed by Sarnak \cite{Sa} this is equivalent to saying that the Liouville subshift $X_{\lio}\subset\{-1,1\}^{\N}$ is actually the full shift and $\lio$ is a generic point
for the Bernoulli measure $(1/2,1/2)^{\otimes{\N}}$, which is the Haar measure on $\{-1,1\}^{\N}$.
When we consider a more general multiplicative function $\bfu:\N\to\bs^1$ which is {\bf aperiodic} (i.e.\ its mean along any arithmetic progression exists and equals zero), and if all powers $\bfu^k$ ($k\geq1$) are still aperiodic, then the analog of~\eqref{chowla1} for $\bfu$ becomes
$$\lim_{N\to\infty}\frac1N\sum_{1\le n\le N}
\bfu^{r_1}(n+a_1)\cdot\ldots\cdot\bfu^{r_k}(n+a_k)
\overline{\bfu^{s_1}(n+b_1)\cdot\ldots\cdot\bfu^{s_\ell}(n+b_\ell)}=0$$
for all powers $r_u,s_t\in\N$ and $\{a_1,\ldots,a_k\}\cap \{b_1,\ldots, b_{\ell}\}=\emptyset$, which in turn means that $\bfu$ is a generic point for $\left(\Leb_{\bs^1}\right)^{\otimes{\N}}$, the Haar measure on $(\bs^1)^{\N}$.

 This more general form of Chowla conjecture is still a particular case of Elliott's conjecture \cite{El1}, \cite{El2}, \cite{El3} which deals with several (possibly different) multiplicative functions (one of which is aperiodic).
Similar conjectures can be formulated for multiplicative functions taking values in the unit disk $\D$, but in this case we have to consider properties of $\bfu$ \emph{relative to} $|\bfu|$, where the latter point is always generic for an ergodic measure (with respect to the left shift), often called the Mirsky measure, on $[0,1]^{\N}$ (see \cite{Sa} for the case of the M\"obius function $\mob$ or e.g.\ \cite{Be-Ku-Le-Ri} for a general case).

In \cite{Ma-Ra-Ta}, Matom\"aki, Radziwi\l\l \ and Tao gave a class of examples of  multiplicative and aperiodic\footnote{Their functions are totally aperiodic, that is, all powers $\bfu^k$, $k\geq1$ are also aperiodic.}  $\bfu:\N\to \bs^1\cup\{0\}$ for which the sequence $$\Big(\frac1N\sum_{1\leq n\leq N}\bfu(n)\overline{\bfu(n+1)}\Big)_{N\geq1}$$ does not converge to zero. This disproved the general form of Chowla conjecture for aperiodic $\bfu$, and in particular,  they disproved Elliott's conjecture. In their class $|\bfu|=\mob^2$ but
they also explained that their example could be modified to get a completely multiplicative $\bs^1$-valued aperiodic function. Finally, this lead them to reformulate Elliott's conjecture in the sense that it is expected to be valid for a subclass of aperiodic functions, the so-called  \emph{strongly aperiodic functions}, and till today this stronger form of Elliott's conjecture is open and under intensive study \cite{Fr}, \cite{Fr-Ho}, \cite{Ma}, {\cite{Ta}-\cite{Ta-Te}.

In this note, we will have a closer look at the counterexample given by Matom\"aki, Radziwi\l\l \
and Tao. To avoid some technical, rather notational, issues,
we will  deal with the completely multiplicative, $\bs^1$-valued
version of their construction, which we call here \emph{the MRT class}
(see the complete description of such multiplicative functions
in Section~\ref{sec:MRTclass}).
If $\bfu\in {\rm MRT}$ then it cannot be generic for the
Haar measure on $(\bs^1)^{\N}$, but still we can ask for which
measures on $(\bs^1)^{\N}$ it is quasi-generic. Each such measure
yields a so-called measure-theoretic \emph{Furstenberg system}
(see Section~\ref{s:Furstenberg}).
In particular, the arguments given in~\cite{Ma-Ra-Ta}
prove that there exists an increasing sequence of integers giving
rise to a Furstenberg system which is measure-theoretically
isomorphic to the action of the identity map on $\bs^1$ equipped
with the Lebesgue measure. What kind of other dynamical systems
can be obtained as Furstenberg systems for $\bfu$ in the MRT class
is a natural question. Furthermore, in the topological setting,
$\bfu$ determines a subshift $X_{\bfu}\subset(\bs^1)^{\N}$ and we
can ask for its topological entropy. Our aim is to prove the
following result.

\vspace{2ex}

 \noindent
{\bf Main Theorem.} {\em Let $\bfu$ be in the MRT class. Then, for each $d\geq 0$, there is a Furstenberg system $(X_{\bfu},\nu_d,S)$ of $\bfu$ which is measure-theoretically isomorphic to the unipotent system
$$
(x_d,x_{d-1},\ldots,x_0)\mapsto (x_d,x_{d-1}+x_d,\ldots, x_0+x_{1})$$
on $\T^{d+1}$ equipped with the $(d+1)$-dimensional Lebesgue measure.

Furthermore,  the Bernoulli shift $\left((\bs^1)^{\N},(\Leb_{\bs^1})^{\otimes\N},S\right)$ is also a Furstenberg system of $\bfu$,  i.e.\
\beq\label{char2}
\mbox{the analog of the Chowla conjecture holds for $\bfu$ along a subsequence.}\eeq In particular, \beq\label{char3}
X_{\bfu}=(\bs^1)^{\N},~\footnote{Cf.\ this equality with a remark after Theorem~10 in~\cite{Sa} as a topological instance of validity of Chowla conjecture.}\eeq
and
\beq\label{char1}
h_{\rm top}(X_{\bfu},S)=\infty.\eeq}

\vspace{2ex}

The proof of the above theorem is given in Section~\ref{s:proof}.
It seems also that this result makes it legitimate to ask whether (some of) properties \eqref{char2}-\eqref{char1} are valid for all totally aperiodic multiplicative functions $\bfu$, $|\bfu|=1$. Note also that the theorem above sheds light on Frantzikinakis' question: see  Problem~3.1 on the list of problems in \cite{Workshop} of whether the system $(x,y)\mapsto (x,x+y)$ on $\T^2$ can be a Furstenberg system of the Liouville function. The expected answer is of course negative (as the positive answer is in conflict with the Chowla conjecture) but, as our result shows, this unipotent system is a Furstenberg system for a class of multiplicative aperiodic functions. Moreover, our Main Theorem yields positive answers to questions raised in Problem~7.3 \cite{Workshop} in the class of MRT functions. More than that, since the ergodic components of $\nu_d$, $d\geq2$, are affine extensions of irrational rotations, the Main Theorem disproves Conjecture~2 from \cite{Fr-Ho3}. 

As each $\bfu\in {\rm MRT}$ is given ``locally by Archimedean
characters'', in Section~\ref{s:Archimedes}, we will deal with
Furstenberg systems of Archimedean characters themselves and will
describe their Furstenberg systems.  Moreover, we will show that no
$\bfu\in\MRT$ satisfies the analog of Sarnak's conjecture. Finally,
in Section~\ref{s:rozne} we show that the analog of
logarithmic Chowla conjecture for $\bfu$ (cf.~\eqref{char2}) holds along a subsequence and discuss further properties of Furstenberg systems of MRT arithmetic functions: strong stationarity and the absence of zero mean on typical short interval.

In Appendix we discuss Furstenberg systems given by $\nu_d$ from the pure  ergodic theory point of view.
We explain their connections with the classical theory of
transformations with quasi-discrete spectrum and give a new characterization of quasi-eigenfunctions which seems to be of independent interest.

\section{Furstenberg systems of a bounded arithmetic function}\label{s:Furstenberg}
Assume that $u:\N\to\C$ is an arithmetic function, $|u|\leq 1$.
Denote by $\mathbb{D}$ the unit disk.
On the space $M\left(\mathbb{D}^{\N}\right)$ of Borel probability measures on $\mathbb{D}^{\N}$, we consider the weak$^\ast$-topology, in which $\nu_m\to\nu$ if and only if $\int_{\mathbb{D}^{\N}}f\,d\nu_m\to \int_{\mathbb{D}^{\N}} f\,d\nu$ for each $f\in C(\mathbb{D}^{\N})$. This topology turns $M\left(\mathbb{D}^{\N}\right)$ into a compact metrizable space.
Let $S$ denote the shift map on $\mathbb{D}^{\N}$, and consider in $M\left(\mathbb{D}^{\N}\right)$ the sequence
\[
\left(E_N(u)\right)_{N\geq 1} := \left(\frac1N\sum_{0\le n< N}\delta_{S^nu}\right)_{N\geq 1}
\]
of empiric probability measures. By compactness, we can choose a converging subsequence
\[
E_{N_m}(u)\to\nu,
\]
and we say that $u$ is \emph{quasi-generic for $\nu$ along the sequence $(N_m)$}.
It is not hard to see (\textit{cf}.\ the Krylov-Bogolyubov theorem) that such a limit point $\nu$ is $S$-invariant. Moreover,
such a $\nu$ is always supported on the subshift $X_u$ generated by $u$,
that is
\[
 X_u := \overline{\{S^n u:n\ge 0\}}.
\]

The measure-theoretic dynamical system $(X_u,\nu,S)$ is called a {\em Furstenberg system} of $u$. Denote by $V(u)$ the set of all probability measures on $X_u$ for which $u$ is quasi-generic:
$$V(u):=\left\{\nu\in M(X_u,S):\:\nu=\lim_{m\to\infty}E_{N_m}(u)
\text{ for some }N_m\to\infty\right\}.$$
Classically, we have the following:
\begin{Prop}\label{p:wlasnosci} (\cite{Denker})
(i) $V(u)\subset M(X_u,S)$ is closed (in the weak$^\ast$-topology).\\
(ii) $V(u)$ is connected, whence
either $|V(u)|=1$ or $V(u)$ is uncountable.
\end{Prop}
Let $Z_0:\mathbb{D}^{\N}\to\mathbb{D}$ be the $1$-coordinate map: $Z_0(y)=y(1)$ for $y=\bigl(y(n)\bigr)_{n\in\N}\in \mathbb{D}^{\N}$. In general, we define $Z_n:=Z_0\circ S^{n}$ as the map $y\mapsto y(n+1)$.
Then, given $\nu\in V(u)$, we obtain a stationary process
$(Z_0,Z_1,\ldots)$  with values in $\mathbb{D}^{\N}$ whose distribution is $\nu$.
Let $\kappa=(Z_0)_\ast(\nu)$ be the distribution of the random variable $Z_0$ under $\nu$.
Then by the $S$-invariance of $\nu$, $\kappa$ is also the distribution of each coordinate $Z_n$, $n\ge0$. An example of particular interest
corresponds to the i.i.d.\ case, arising when $\nu$ is a product measure, {\em i.e.}, when $\nu$ is of the form $\kappa^{\otimes \N}$.  We have the following observation.

\begin{Prop}
\label{c:rou10}
Let $\kappa$ be a probability measure on $\mathbb{D}$.
Assume that for each $d\geq 0$ there exists $\nu_d\in V(u)$
under which the distribution of $(Z_0,\ldots,Z_{d})$ is $\kappa^{\otimes d+1}$.
Then $\kappa^{\ot\N}\in V(u)$.
\end{Prop}

\begin{proof} By compactness, we can assume that $\nu_d\to\rho$, and by Proposition~\ref{p:wlasnosci}~(i), $\rho\in V(u)$.
All we need to show is that
$$
\mathbb{E}_{\rho} (Z_0^{q_1}\circ S^{j_1}\cdot \ldots\cdot Z_0^{q_k}\circ S^{j_k})=\prod_{i=1}^k \mathbb{E}_\kappa Z_0^{q_i}$$
for each $k\geq1$, $q_i\in\Z$ and $0\le j_1<\ldots <j_k$. But the equality above is true if $\rho$ is replaced by $\nu_d$ for $d\geq j_k$, and since
$Z_0^{q_1}\circ S^{j_1}\cdot \ldots\cdot Z_0^{q_k}\circ S^{j_k}\in C(X_u)$,  the result follows.
\end{proof}

%

\begin{Remark}
\label{r:entropy}
If $\kappa^{\ot\N}\in V(u)$ then by the variational principle (see e.g.\ \cite[Section 8.2]{walters})  $h_{\rm top}(X_u,S)\geq h(X_u,\kappa^{\ot\N},S)=H(\kappa)$. If the distribution of $\kappa$ is continuous then immediately $H(\kappa)=+\infty$, whence
$h_{\rm top}(X_u,S)=+\infty$ in this case.
\end{Remark}

Proposition~\ref{c:rou10} can be useful if we want to show that the product measure yields a Furstenberg system of Bernoulli type (``Chowla holds along a subsequence''). Indeed, we only need to show the existence of Furstenberg systems which yield some finite degree of independence of the process $(Z_n)_{n\ge0}$ and such can be firstly of zero entropy and even very non-ergodic (i.e.\ belonging to ${\rm Erg}^\perp$), cf.\ also \cite{Fl}. Together with Remark~\ref{r:entropy}, it also gives a nice criterion to show that the topological entropy of $u$ is infinite.

\section{MRT multiplicative functions}
\label{sec:MRTclass}

In this section we describe more precisely the MRT class of completely multiplicative functions to which our Main Theorem applies.
We start by giving a formal definition of this class, then we resume the construction given in~\cite{Ma-Ra-Ta} by
Matom\"aki, Radziwi\l\l \ and Tao of a completely multiplicative function meeting the required property. Finally, we give the key property of MRT arithmetic functions that will be used in the proof of our Main Theorem.

\begin{Def}
 \label{def:MRT}
 A completely multiplicative function $\bfu:\N\to\bs^1$ belongs to the MRT class if there exist
 two increasing sequences of integers $(t_m)$ and $(s_m)$ such that, for each $m\ge1$, we have
 the following properties:
 \begin{align}
  &\bullet\ t_m<s_{m+1}<s_{m+1}^2\le t_{m+1}, \nonumber\\
  &\bullet\ \text{for each prime }p\in (t_m,t_{m+1}],\ \bfu(p)=p^{is_{m+1}}, \label{eq:MRT1}\\
  &\bullet\ \text{for each prime }p\le t_m,\ \left|\bfu(p)-p^{is_{m+1}}\right|<\frac{1}{t_m^2}. \label{eq:MRT2}
 \end{align}
\end{Def}

Here is the method to get such a function.
We just have to define $\bfu(p)$ for each prime $p$ and to construct the sequences $(t_m)$ and $(s_m)$,
which is done inductively as follows.
Start by choosing an integer $t_1\in\N$ and set, for each prime $p\le t_1$, $\bfu(p):=1$.
Now, assume that for some $m\ge1$ we have already defined $t_m$ and $\bfu(p)$ for each $p\le t_m$.  In the Cartesian product $\prod_{p\le t_m}\bs^1$, we consider the sequence of points
\[
 \Bigl(\left(p^{i s}\right)_{p\le t_m}\Bigr)_{s\in\N}.
\]
Since the numbers $\log p$, $p\le t_m$, are linearly independent over the integers, this sequence is dense in $\prod_{p\le t_m}\bs^1$. Thus, we can choose 
$s_{m+1}>t_m$ so that~\eqref{eq:MRT2} is satisfied.
We then choose $t_{m+1}\ge s_{m+1}^2$, and for $t_m<p\le t_{m+1}$ we set $\bfu(p):=p^{i s_{m+1}}$.
Doing this inductively for each $m\ge1$ gives a completely multiplicative function $\bfu\in \MRT$.

\begin{Remark} It is also interesting to note that the growth
of $s_{m+1}/t_m$ is necessarily superpolynomial: in fact, it follows from
Propositions~\ref{prop:key} and~\ref{prop:final} that, for each $\beta>0$,   $t_m<s_{m+1}^\beta$
for $m$ large enough.

It is also shown in \cite{Ma-Ra-Ta} that once $s_{m+1}>e^{t_m}$ for $m\geq1$,  the resulting $\bfu$ is aperiodic.
\end{Remark}

\medskip

We will use the following easy lemma.

\begin{Lemma}
 \label{lemma:easy}
 Let $\bfu\in\MRT$ and $(t_m)$, $(s_m)$ be as in Definition~\ref{def:MRT}. Let $m\ge1$ and $n\le t_{m+1}$.
If the number of prime factors of $n$ less than or equal to $t_m$ (counting multiplicity) is bounded by $t_m$,
then
\begin{equation}
 \label{eq:most_n}
 \left|\bfu(n)-n^{is_{m+1}} \right| \le \frac{1}{t_m}.
\end{equation}
\end{Lemma}

\begin{proof}
We write $n$ as a product of primes
\[ n = \prod_{p\le t_{m}} p^{\alpha_p(n)}\prod_{t_m<p\le t_{m+1}} p^{\alpha_p(n)}. \]
By the complete multiplicativity of $\bfu$ and by~\eqref{eq:MRT1}, we then have
\[ \bfu(n) = \prod_{p\le t_{m}} \bfu(p)^{\alpha_p(n)}  \left(\prod_{t_m<p\le t_{m+1}} p^{\alpha_p(n)}\right)^{is_{m+1}}. \]
Using~\eqref{eq:MRT2} in the first product above, we get
\[ \left|\bfu(n)-n^{is_{m+1}} \right| \le \frac{\sum_{p\le t_{m}} \alpha_p(n)}{t_m^2}. \]
In particular, if $\sum_{p\le t_{m}} \alpha_p(n)\le t_m$, then
\[ \left|\bfu(n)-n^{is_{m+1}} \right| \le \frac{1}{t_m}. \]
\end{proof}

The next lemma is useful to estimate the density of the integers $n$ for which~\eqref{eq:most_n} is not valid.

\begin{Lemma}
  \label{lemma:Bt}
  For $t\in\N$, denote
  \[
    B_t:=\{n\geq1:\:\sum_{p\leq t
    }\alpha_p(n)\geq t\},
  \]
  i.e.\ $B_t$ is the set of integers $n\ge1$ having at least $t$ prime factors less than or equal to $t$ (counting multiplicity).
  Then
  \begin{equation}
   \label{eq:Bt}
   \varepsilon_t := \sup_{N\ge1} \frac{1}{N} \sum_{1\le n\le N} \ind{B_t}(n) \tend{t}{\infty} 0.
  \end{equation}
 \end{Lemma}

 \begin{proof}
  For each $t\in\N$, set $k_t:= \left\lfloor \frac{t}{\pi(t)} \right\rfloor$ (where $\pi(t)$ denotes, as usual, the number of primes up to $t$).
  Let $n\in\N$; if, for each prime $p\le t$, $p^{k_t}$ does not divide $n$, then the number of prime factors of $n$ less than or equal to $t$ (counting multiplicity) is at most $\pi(t)(k_t-1)<t$, hence $n\notin B_t$.
  By contraposition, if $n\in B_t$, there exists a prime number
  $p\le t$ such that $p^{k_t}|n$. Therefore, we have for each $N\ge1$
  \begin{align*}
   \frac{1}{N} \sum_{1\le n\le N} \ind{B_t}(n) &\le  \frac{1}{N} \sum_{1\le n\le N} \sum_{p\le t} \ind{p^{k_t}|n} \\
            & =   \sum_{p\le t} \frac{1}{N} \sum_{1\le n\le N} \ind{p^{k_t}|n} \\
            &  \leq   \sum_{p\le t} \frac{1}{p^{k_t}} \\
            & <  \sum_{p}  \frac{1}{p^{k_t}} \tend{t}{\infty} 0 \quad\text{(since $k_t\to \infty$ as $t\to\infty$).}
  \end{align*}
 \end{proof}
 Using Lemma~\ref{lemma:easy} and Lemma~\ref{lemma:Bt}, we obtain the following result.
\begin{Prop}
 \label{prop:MRT}
 Let $\bfu\in\MRT$ and $(t_m)$, $(s_m)$ be as in
 Definition~\ref{def:MRT}. Let
 $(N_m)$ be an increasing sequence of integers with
 $N_m\le t_{m+1}$ for each $m\in\N$.
 Then
 \[
    \frac{1}{N_m}\# \left\{ n\in\{1\ldots,N_m\} :  \left| \bfu(n)-n^{is_{m+1}} \right| > \frac{1}{t_m} \right\} \le  \varepsilon_{t_m}  \tend{m}{\infty} 0.
 \]

\end{Prop}

\section{Proof of the Main Theorem} \label{s:proof}

Here is an outline of the proof. First, we present in Section~\ref{sec:processes} a family of stationary processes, taking values in the unit circle, parametrized by an integer $d\geq0$. Such a process generates a measure-theoretic dynamical system isomorphic to the unipotent system appearing in the statement of the theorem. It is easily characterized by two conditions: each coordinate of the process is uniformly distributed on the circle, and some deterministic function $\phi_{d+1}$ of the coordinates of the process is constantly equal to 1. We completely describe in Proposition~\ref{prop:nud} and Remark~\ref{rem:nud} the distribution $\nu_d$ of such a process, in particular we show that $d+1$ consecutive coordinates are independent. Then Proposition~\ref{prop:key} provides a criterion for an arithmetic function $u$ to be quasi-generic for this probability measure $\nu_d$. This criterion involves the functions $\phi_{d+1}$, and we study in Section~\ref{sec:polynomials} two sequences of polynomials related with this family of functions. We show in Section~\ref{s:dowodMRT} that for  $\bfu\in \MRT$, the criterion is fulfilled for each $d\geq0$. For this we use Proposition~\ref{prop:MRT} which allows us to replace, in the appropriate range, $\bfu(n)$ by
$n^{is_{m+1}}$. The criterion then becomes an evaluation of some exponential sums, that we can handle with the help of Kusmin-Landau Theorem (Theorem~\ref{t:KuLa}). Finally, using Proposition~\ref{c:rou10}, we can conclude that $\bfu$ is also quasi-generic for the product measure $\bigl(\Leb_{\mathbb{S}^1}\bigr)^{\otimes \mathbb{N}}$.

 \subsection{Processes in the unit circle}
 \label{sec:processes}

We recall the notation $Z_n=Z_0\circ S^n$ ($n\ge0$) from Section~\ref{s:Furstenberg}, but we restrict ourselves to the set of sequences taking values in $\bs^1$.
We define inductively a new sequence of maps $(\Yg{d})_{d\ge0}$ defined on $(\bs^1)^\N$
 taking also their values in $\bs^1$, by $\Yg{0}:=Z_0$, $\Yg{1}:=Z_1/Z_0$, and in general for each $d\ge0$,
 \[ \Yg{d+1} := \frac{\Yg{d}\circ S}{\Yg{d}}. \]
 We also define the auxilliary sequence $(X_d)_{d\ge0}$ taking values in the one-dimensional torus $\T:=\R/\Z$, by
 \[ e^{i2\pi X_d} := Y^{(d)}. \]
 As $e^{i2\pi X_{d+1}}=\Yg{d+1}=e^{i2\pi (X_d\circ S-X_d)}$, we get
 \begin{equation}
 \label{eq:XdoT}
  \forall d\ge0,\quad X_d\circ S = X_{d+1}+X_d.
 \end{equation}
 Moreover, by an easy induction on $n$ using the above formula, we can recover the process each $Z_n$ from $(X_d)$ by the relation
 \begin{equation}
  \label{eq:ZfromX}
  \forall n\ge0,\quad Z_n = e^{i2\pi\sum_{d=0}^n {{n}\choose{d}} X_d}.
 \end{equation}

\begin{Prop}
\label{prop:nud}
Let $\nu$ be a shift-invariant probability measure on $(\bs^1)^\N$.
 Assume that there exists $d\ge0$ such that, under $\nu$,
 \begin{itemize}
  \item $\Yg{d}$ is $S$-invariant (\textit{i.e.} $\Yg{d+1}=1$ $\nu$-a.s.),
  \item the distribution of $\Yg{d}$ is uniform on $\bs^1$.
 \end{itemize}
 Then, for each $n$, $Z_n$ is uniformly distributed on $\bs^1$, and $Z_0,Z_1,\ldots,Z_d$ are independent.
\end{Prop}

 \begin{proof}
  As $Y^{(d)}$ is $S$-invariant, so is $X_d$. Hence, in  every ergodic component of $\nu$, $X_d$ is a.s.\ constant. Moreover, since the distribution of $X_d$ under $\nu$ is the Lebesgue measure $\Leb_{\T}$ on $\T$, in almost every ergodic component the value of $X_d$ is irrational. Let us fix such an ergodic component, say $\tilde \nu$, and let $\alpha$ be the value taken by $X_d$ under $\tilde \nu$.
  From~\eqref{eq:XdoT}, we get
  \begin{align*}
   \left( X_{d-1},\ldots,X_0 \right) \circ S &= \left( X_{d-1}+X_d,X_{d-2}+X_{d-1},\ldots,X_0 + X_1\right)\\
   &= \left( X_{d-1}+\alpha,X_{d-2}+X_{d-1},\ldots,X_0 + X_1\right).
  \end{align*}
  But when $\alpha$ is irrational, the map
  \[
     \left( x_{d-1},\ldots,x_0 \right) \in\T^d \mapsto \left( x_{d-1}+\alpha,x_{d-2}+x_{d-1},\ldots,x_0 + x_1\right)
  \]
  is uniquely ergodic \cite{Fu}, with the $d$-fold product $(\Leb_{\T})^{\otimes d}$ as the only ergodic invariant measure.
  We deduce that, under $\tilde \nu$, the distribution of $\left( X_{d-1},\ldots,X_0 \right)$ must be $(\Leb_{\T})^{\otimes d}$. Integrating over the set of ergodic components, we get that under $\nu$, the distribution of $\left( X_d, X_{d-1},\ldots,X_0 \right)$ is $(\Leb_{\T})^{\otimes (d+1)}$. Then, from Formula~\eqref{eq:ZfromX}, we can write each $0\le n\le d$,
  \[ Z_n=\tilde Z_{n-1} e^{i2\pi X_n},\]
  where $\tilde Z_{n-1}$ is measurable with respect to $X_0,\ldots,X_{d-1}$ ($\tilde Z_{n-1}=1$ for $n=0$). From this it follows that $Z_n$ is uniformly distributed on $\bs^1$ conditionally to $X_0,\ldots,X_{n-1}$. This concludes the proof.
 \end{proof}

\begin{Remark}
\label{rem:nud}
 The proof shows in fact that, for each $d\ge0$, there is a unique shift-invariant measure $\nu_d$ on $(\bs^1)^\N$ such that whenever the assumptions of Proposition~\ref{prop:nud} are satisfied, then the distribution of the stationary process $(Z_n)_{n\in\N}$ is $\nu_d$. Under $\nu_d$, any $d+1$ consecutive coordinates of the process are independent and uniformly distributed on the circle, and for each $n\ge0$ the coordinate $Z_{n+d+1}$ is a deterministic function of $Z_n,\ldots,Z_{n+d}$ given by the condition $\Yg{d+1}=1$. Moreover, looking at the process $\bigl((X_0,X_1,\ldots,X_d)\circ S^n\bigr)_{n\ge0}$, we see that the dynamical system $\bigl((\bs^1)^\N,\nu_d,S\bigr)$ is measure-theoretically isomorphic to the unipotent system described in the statement of the Main Theorem.
\end{Remark}

\medskip

It will be useful to introduce, for each $d\ge0$, the function $\ph{d}:(\bs^1)^{d+1}\to \bs^1$ such that $\Yg{d}=\ph{d}(Z_0,\ldots,Z_d)$. For $d=0,\ldots,4$, these functions are given by
\begin{equation}
\label{eq:phid}
 \begin{split}
  \ph{0}(z_0) = \frac{z_0}{1}\quad ;\quad\ph{1}(z_0,z_1) = \frac{z_1}{z_0}\quad ;\quad \ph{2}(z_0,z_1,z_2) = \frac{z_0z_2}{z_1^2}\quad ;\\
  \ph{3}(z_0,z_1,z_2,z_3) = \frac{z_1^3 z_3}{z_0 z_2^3}\quad ;\quad \ph{4}(z_0,z_1,z_2,z_3,z_4) = \frac{z_0 z_2^6 z_4}{z_1^4 z_3^4}.
 \end{split}
\end{equation}

 In general, we can always write the function $\ph{d}$ as a quotient
 \beq\label{e:3a} \ph{d}(z_0,\ldots,z_d) = \frac{\pi_d(z_0,\ldots,z_d)}{\tilde\pi_d(z_0,\ldots,z_d)} ,\eeq
 where $\pi_d(z_0,\ldots,z_d)$ and $\tilde\pi_d(z_0,\ldots,z_d)$ are products of some $z_j$'s (with possible multiplicity). These sequences of products are completely defined by the following recurrence property: $\pi_0(z_0):=z_0$, $\tilde\pi_0(z_0):=1$, and for each $d\ge1$,
 \begin{equation}
  \label{eq:defpi1}
 \pi_{d+1}(z_0,\ldots,z_{d+1}):= \pi_d(z_1,\ldots,z_{d+1})\,\tilde\pi_d(z_0,\ldots,z_d),
 \end{equation}
 and
 \begin{equation}
  \label{eq:defpi2}
 \tilde\pi_{d+1}(z_0,\ldots,z_{d+1}):= \tilde\pi_d(z_1,\ldots,z_{d+1})\,\pi_d(z_0,\ldots,z_d).
 \end{equation}

 Note that, as $\pi_d$ and $\tilde\pi_d$ only involve a product of variables, their definition can be extended to $\CC^{d+1}$.

 \medskip

 The next proposition is a key ingredient for the
 identification of some Furstenberg systems of
 multiplicative functions.

 \begin{Prop}
  \label{prop:key}
  Let $u=\bigl(u(n)\bigr)_{n\in\N}\in(\bs^1)^\N$. Assume that, for some $d\ge1$ and some
  increasing sequence $(N_m)\subset\N$, we have
  \begin{equation}
   \label{eq:hyp1}
   \frac{1}{N_m} \sum_{n\le N_m} \ph{d+1}\bigl(u(n),u(n+1),\ldots,u(n+d+1)\bigr) \tend{m}{\infty} 1,
  \end{equation}
  and that
\begin{equation}
   \label{eq:hyp2}
   \forall\ell\ge1,\ \frac{1}{N_m} \sum_{n\le N_m} \ph{d}^\ell\bigl(u(n),u(n+1),\ldots,u(n+d)\bigr) \tend{m}{\infty} 0.
  \end{equation}
  Then, along the subsequence $(N_m)$, $u$ is quasi-generic for the measure $\nu_d$ described in Remark~\ref{rem:nud}.
 \end{Prop}

 \begin{proof}
  Let $\mu$ be a shift-invariant measure on $(\bs^1)^\N$ such that, along some subsequence of $(N_m)$, $u$ is quasi-generic for $\mu$.
  From~\eqref{eq:hyp1}, we get that
  \[ 1=\EE_\mu\left[ \ph{d+1}(Z_0,\ldots,Z_{d+1}) \right] = \EE_\mu\left[ \Yg{d+1} \right],
  \]
  hence $\Yg{d+1}=1$ $\mu$-a.s. And from~\eqref{eq:hyp2}, we get that for each $\ell\ge1$,
  \[ 0=\EE_\mu\left[ \ph{d}^\ell(Z_0,\ldots,Z_{d}) \right] = \EE_\mu\left[ (\Yg{d})^\ell \right],
  \]
  which shows that, under $\mu$, $\Yg{d}$ is uniformly distributed on $\bs^1$. Then, with Proposition~\ref{prop:nud} and Remark~\ref{rem:nud}, we conclude that $\mu = \nu_d$.
 \end{proof}

 \subsection{Special sequences of polynomials}
 \label{sec:polynomials}

 We now introduce two sequences $(P_d)_{d\ge0}$ and $(Q_d)_{d\ge0}$ of polynomials in the variable $n$, linked to the analysis of the preceding section by the following definition: for each $d\ge0$,
 \begin{equation}
  \label{eq:defPQ}
  P_d(n):=\pi_d(n,n+1,\ldots,n+d),\quad\text{and}\quad Q_d(n):=\tilde\pi_d(n,n+1,\ldots,n+d).
 \end{equation}
The first polynomials of this family are given below (compare with the numerators and denominators in~\eqref{eq:phid}).
\begin{align*}
    P_0(n)=n\ &;\ Q_0(n)=1\ ;\\
    P_1(n)=n+1\ &;\ Q_1(n)=n\ ;\\
    P_2(n)=n(n+2)\ &;\ Q_2(n)=(n+1)^2\ ;\\
    P_3(n)=(n+1)^3 (n+3)\ &;\ Q_3(n)=n (n+2)^3\ ;\\
    P_4(n)=n (n+2)^6 (n+4)\ &;\ Q_4(n)=(n+1)^4 (n+3)^4\ .
\end{align*}
Note that, according to \eqref{eq:defpi1} and \eqref{eq:defpi2}, these sequences of polynomials satisfy the following recurrence property:
 \begin{align}
 \label{eq:induction_PQ}
  \begin{split}
  P_{d+1}(n) &= P_d(n+1) Q_d(n),\\
  Q_{d+1}(n) &= Q_d(n+1) P_d(n).
 \end{split}
 \end{align}
%
%
%

\begin{Prop}
  \label{prop:polynomials}
  For each $d\ge1$, $P_d$ and $Q_d$ are both of degree $2^{d-1}$, and the degree of the difference $R_d:=Q_d-P_d$ is $2^{d-1}-d$.
 \end{Prop}

 The computation of the degree of $P_d$ and $Q_d$ is a straightforward induction using~\eqref{eq:induction_PQ}. For the degree of $R_d$, we will use the following lemma.

 \begin{Lemma}
  \label{lemma:degree_of_R}
  Let $P$ and $R$ be two polynomials, of degree $p$ and $r$, respectively and with $p\neq r$. Let $\tilde R$ be the polynomial defined by
  \[  \tilde{R}(n) := R(n+1)P(n)-R(n)P(n+1). \]
  Then the degree of $\tilde R$ is $r+p-1$.
 \end{Lemma}

 \begin{proof}
  Let us write the two terms of highest degree in $P$ and $R$:
  \[ P(n) = a_P n^p + b_P n^{p-1}+\cdots \]
  and
  \[ R(n) = a_R n^r + b_R n^{r-1}+\cdots \]
  where $a_P a_R \neq 0$.
  Then a direct computation shows that the two terms of highest degree in $R(n+1)P(n)$ are
  \[ R(n+1)P(n) = a_R\,a_P\, n^{r+p}
    + \left( a_R\,b_P+a_P\,b_R+r\,a_P\,a_R \right) n^{r+p-1}+\cdots
  \]
Likewise, the two terms of highest degree in $R(n)P(n+1)$ are
  \[ R(n)P(n+1) = a_R\,a_P\, n^{r+p}
    + \left( a_R\,b_P+a_P\,b_R+p\,a_P\,a_R \right) n^{r+p-1}+\cdots
  \]
  It follows that the term of highest degree in $\tilde R(n)$ is $(r-p)a_P\,a_R\,n^{r+p-1}$.
 \end{proof}

 \begin{proof}[End of the proof of Proposition~\ref{prop:polynomials}]
  We show by induction on $d\ge1$ that the degree of $R_d$ is $2^{d-1}-d$. This is already true for $d=1$ since $R_1=-1$ is of degree 0.
  Then, assume that the result holds for some $d\ge1$, and let us consider the polynomial $R_{d+1}$. We have
  \begin{align*}
   R_{d+1}(n) &= Q_{d+1}(n) - P_{d+1}(n) \\
            &= Q_d(n+1)P_d(n)-P_d(n+1)Q_d(n)\\
            &= \bigl( P_d(n+1)+R_d(n+1) \bigr) P_d(n) - P_d(n+1) \bigl( P_d(n)+R_d(n) \bigr)\\
            &= R_d(n+1)  P_d(n) - R_d(n) P_d(n+1).
  \end{align*}
  But we know that the degree of $P_d$ is $2^{d-1}$, and by the induction assumption the degree of $R_d$ is $2^{d-1}-d$. These degrees are different, therefore Lemma~\ref{lemma:degree_of_R} applies, and shows that the degree of $R_{d+1}$ is $2^d-(d+1)$.

 \end{proof}

 \subsection{Furstenberg systems of MRT multiplicative functions}
 \label{s:dowodMRT}

 Now, we consider a completely multiplicative arithmetic function $\bfu=\bigl(\bfu(n)\bigr)_{n\ge1}\in\MRT$. Let $(t_m)$ and $(s_m)$ be the associated sequences, as in Definition~\ref{def:MRT}.
 Let $(N_m)$ be an increasing sequence of integers with $N_m\le t_{m+1}$, and let $d\ge0$ and $\ell\ge1$ be fixed integers.
  Our purpose is to apply Proposition~\ref{prop:key} to $\bfu$, therefore we have to evaluate the expressions of the form
 \begin{equation}
 \label{eq:form}
   \frac{1}{N_m} \sum_{1\le n\le N_m} \ph{d}^\ell\bigl(\bfu(n),\bfu({n+1}),\ldots,\bfu({n+d})\bigr).
  \end{equation}

By Proposition~\ref{prop:MRT}, we have
 \begin{align}
 & \frac{1}{N_m} \sum_{1\le n\le N_m} \ph{d}^\ell\bigl( \bfu(n),\ldots,\bfu(n+d)\bigr)\nonumber \\
 & = \frac{1}{N_m} \sum_{1\le n\le N_m} \ph{d}^\ell\bigl(n^{is_{m+1}},\ldots,(n+d)^{is_{m+1}}\bigr) + {\rm o}(1) \nonumber \\
 & = \frac{1}{N_m} \sum_{1\le n\le N_m} e^{i\ell s_{m+1} \log \frac{P_d(n)}{Q_d(n)}} + {\rm o}(1).
 \label{eq:exp}
 \end{align}
 Therefore, in view of applying Proposition~\ref{prop:nud}, we can replace an expression of the form~\eqref{eq:form} by
 \[
  \frac{1}{N_m} \sum_{1\le n\le N_m} e^{i\ell s_{m+1} f_d(n)},
 \]
 where $f_d$ is defined by
  \begin{equation}
  \label{eq:def_fd}
  f_d(x) := \log \frac{P_d(x)}{Q_d(x)}.
 \end{equation}

 Note that, by Proposition~\ref{prop:polynomials}, for $d\ge1$, the $d$ terms of highest degrees in $P_d$ and $Q_d$ coincide. In particular, we have
 \[ \lim_{x\to\infty} \frac{P_d(x)}{Q_d(x)} = 1, \]
so $f_d(x)$ is well defined for $x\in\RR$ large enough. We will use the following results on the asymptotic behaviour of $f_d$.

%
%

\begin{Lemma}
 \label{lemma:fd}
 For each $d\ge1$, there exists $K_d\neq0$ such that
 \begin{equation}
  \label{eq:fd}
  f_d(x) \sim  \frac{K_d}{x^{d}}\quad \text{as }x\to\infty.
 \end{equation}
For  each $d\ge0$, there exists $L_d\neq 0$ such that
 \begin{equation}
  \label{eq:fdprime}
  f'_d(x) \sim  \frac{L_d}{x^{d+1}}\quad \text{as }x\to\infty.
 \end{equation}
Moreover, there exists $H_d>0$ such that $f'_d$ is monotone on $[H_d,+\infty)$.
\end{Lemma}

\begin{proof}
For $d=0$, note that $f_0(x)=\log x$, so that $f'_0(x)=\frac{1}{x}$ and the result concerning $f'_0$ is obvious. We consider now the case $d\ge1$.
From Proposition~\ref{prop:polynomials}, we can write
\[ f_d(n) = \log\frac{P_d(n)}{P_d(n)+R_d(n)} = -\log\left( 1 + \frac{R_d(n)}{P_d(n)} \right), \]
where $\deg P_d=2^{d-1}$ and $\deg R_d=2^{d-1}-d$, which yields~\eqref{eq:fd}.

Deriving $f_d$ gives
 \begin{align*}
f'_d &= \frac{Q_d}{P_d}\cdot \frac{P'_d(P_d+R_d)-P_d(P'_d+R'_d)}{Q_d^2} \\
     &= \frac{P'_dR_d-P_dR'_d}{P_dQ_d}.
 \end{align*}
 Since the degrees of $P_d$ and $R_d$ are different, the degree of the numerator is equal to $\deg P_d+\deg R_d-1= 2^d-(d+1)$. But the degree of the denominator is $\deg P_d+\deg Q_d=2^d$, and this gives~\eqref{eq:fdprime}.
 Finally, as a nonzero rational fraction, the second derivative $f_d''$ has finitely many zeros, from which we get the last claim of the lemma.
\end{proof}

We can now state the precise asymptotics which, together with Proposition~\ref{prop:key}, will allow us to identify some Furstenberg systems of $\bfu$.

 \begin{Prop}
  \label{prop:final}
  Let $d\ge0$ be a fixed integer, and choose a real number $\beta$ such that
  \[
   \frac{1}{d+1} < \beta < \dfrac{1}{d}\text{ if $d\ge1$, }\quad 1 < \beta < 2\text{ if $d=0$.}
  \]
  Set, for each $m\ge1$, $N_m:=\left\lfloor (s_{m+1})^\beta \right\rfloor$. Then
  \begin{equation}
   \label{eq:cas0}
   \forall \ell\ge1, \quad \frac{1}{N_m} \sum_{1\le n\le N_m} e^{i\ell s_{m+1} f_d(n)}\tend{m}{\infty}0,
  \end{equation}
and
  \begin{equation}
   \label{eq:cas1}
   \quad \frac{1}{N_m} \sum_{1\le n\le N_m} e^{i s_{m+1} f_{d+1}(n)}\tend{m}{\infty}1.
   \end{equation}
 \end{Prop}

An important tool in the proof of the above proposition is the following theorem of Kusmin-Landau, which we state as in \cite{Ri}.
Here, if $x$ is a real number, $\|x\|$ stands for the distance of $x$ to the nearest integer.

\begin{Th}[Kusmin-Landau Theorem] \label{t:KuLa}
If $f:[a,b]\to \R$ is $C^1$, $f'$ is monotone and
$\|f'\|\geq \lambda_1>0$ then
$$
\left|\sum_{n\in[a, b]}e^{i2\pi f(n)}\right|\leq\frac{2}{\pi\lambda_1}.$$
\end{Th}

\begin{proof}[Proof of Proposition~\ref{prop:final}]
 Choose $\alpha$ such that $\frac{1}{d+1}<\alpha<\beta$. We can replace the average in~\eqref{eq:cas0} and in~\eqref{eq:cas1} by the average over $s_{m+1}^\alpha\le n\le s_{m+1}^\beta$. For such an $n$, using~\eqref{eq:fd}, we get
 \[ \left| s_{m+1} f_{d+1}(n) \right| = {\rm O} \left(\frac{1}{s_{m+1}^{\alpha(d+1)-1}}\right)\tend{m}{\infty}0,
 \]
and this immediately gives~\eqref{eq:cas1}.

By Lemma~\ref{lemma:fd}, for $m$ large enough $f'_d$ is monotone on $\left[s_{m+1}^\alpha,s_{m+1}^\beta\right]$, and there exists $K>0$ (depending only on $\ell$ and $d$) such that, for $n$ in this interval,
\[ \| \ell s_{m+1} f'_d(n) \| \ge K \frac{s_{m+1}}{n^{d+1}} \ge K\frac{1}{s_{m+1}^{\beta(d+1)-1}}.\]
By Kusmin-Landau Theorem, we get
\[
 \left| \sum_{s_{m+1}^\alpha \le n \le s_{m+1}^\beta} e^{i\ell s_{m+1} f_d(n)} \right| = {\rm O} \left( s_{m+1}^{\beta(d+1)-1}  \right),
\]
and finally
\[ \left|\frac{1}{N_m} \sum_{1\le n\le N_m} e^{i\ell s_{m+1} f_d(n)} \right| = {\rm O} \left( s_{m+1}^{\beta d-1}  \right)\tend{m}{\infty}0.
\]
\end{proof}

\medskip

Putting together~\eqref{eq:exp}, Proposition~\ref{prop:final} and Proposition~\ref{prop:key}, we get the following result.

\begin{Th}
 For any $d\ge0$, the measure-preserving system $\left((\bs^1)^\N,  \nu_d, S\right)$ appears as a Furstenberg system of $\bfu$.
\end{Th}
Using Proposition~\ref{c:rou10} 
and the fact that, under $\nu_d$, $d+1$ consecutive coordinates of the process $(Z_n)$ are independent and uniformly distributed on $\bs^1$, we obtain the following.

\begin{Cor}
 \label{cor:final}
  The system $\left((\bs^1)^\N,  \left(\Leb_{\bs^1}\right)^{\otimes \N},S\right)$ is a Furstenberg system of~$\bfu$.
\end{Cor}
This concludes the proof of the Main Theorem.

\section{Archimedean characters and slowly varying arithmetic functions}\label{s:Archimedes}

We recall that an {\em Archimedean character} is a completely multiplicative function of the form
$n\mapsto n^{it}$
with some fixed $t\in\R$.
As
\[
 \left| (n+1)^{it}-n^{it} \right| = \left|e^{it\log(1+1/n)}-1\right| \tend{n}{\infty} 0,
\]
Archimedean characters fall into the category of \emph{slowly varying} arithmetic functions, that are bounded arithmetic functions $u$
satisfying
 \[
  u(n+1)-u(n)\tend{n}{\infty}0.
 \]
A useful weakening of this property is the following: we say that the bounded arithmetic function $u$
is \emph{mean slowly varying} if
 \begin{equation}
 \label{eq:mean-sv}
  \frac{1}{N}\sum_{n\le N}|u(n+1)-u(n)|\tend{N}{\infty}0.
 \end{equation}
 Note that this is equivalent to the fact that $u(n+1)-u(n)\to0$ on a subset of $n$ of density~1.

 \medskip

It is worth mentioning a result by Klurman~\cite[Theorem~1.8]{klurman2017}, who proved a conjecture by K\'atai concerning mean slowly varying multiplicative functions (non necessarily bounded). From his theorem, it easily follows that the only mean slowly varying multiplicative functions taking values in the unit circle are the Archimedean characters.

\subsection{Furstenberg systems of (mean) slowly varying functions}

\begin{Prop} \label{p:slowlyvar} The arithmetic function $u:\N\to\D$ is mean slowly varying if and only if
all Furstenberg systems of $u$ are measure-theoretically isomorphic to the action of the identity
on some probability space.\end{Prop}
\begin{proof}
Assume first that $u$ is mean slowly varying, and that $u$ is quasi-generic for some $S$-invariant measure $\nu$ on $\D^\N$ along a subsequence $(N_m)$.
Then, by~\eqref{eq:mean-sv}, we have
\[
 \EE_\nu\bigl[|Z_1-Z_0|\bigr] = \lim_{m\to\infty} \frac{1}{N_m}\sum_{n\le N_m}|u(n+1)-u(n)| = 0.
\]
It follows that $Z_1=Z_0$ $\nu$-a.e., and more generally by $S$-invariance, for each $k\in\N$, we also have $Z_{k+1}=Z_k$ $\nu$-a.e. Hence, $\nu$ is concentrated on the subset of sequences with identical coordinates, and $S=\Id$ $\nu$-a.e.

Conversely, assume that~\eqref{eq:mean-sv} fails. Then there exists a subsequence $(N_m)$ such that
\[
  \lim_{m\to\infty} \frac{1}{N_m}\sum_{n\le N_m}|u(n+1)-u(n)|  > 0,
\]
and by compactness of $M(\D^\N)$ we can assume that $u$ is quasi-generic for some $S$-invariant measure $\nu$ along $(N_m)$. But then we have
\[
 \EE_\nu\bigl[|Z_1-Z_0|\bigr] = \EE_\nu\bigl[|Z_0\circ S -Z_0|\bigr] > 0,
\]
and the Furstenberg system of $u$ defined by $\nu$ is not measure-theoretically isomorphic to the action of the identity.
\end{proof}

\begin{Prop}
 \label{prop:slowly}
 Let $u:\N\to\D$ be a slowly varying arithmetic function, and let $X_u\subset\D^\N$ be the subshift generated by $u$. Then the topological entropy of $(X_u,S)$ is zero.
\end{Prop}
\begin{proof}
 We start by observing that, as $u$ is slowly varying, for each $\varepsilon>0$ there exists $L_\varepsilon$ such that whenever $\ell\ge L_\varepsilon$,
 \begin{equation}
  \label{eq:sv1}
  \forall n\in\N,\ \frac{1}{\ell}\sum_{n\le j <n+\ell} \left| u(j+1)-u(j)\right| < \varepsilon.
 \end{equation}
Let $w$ be a sequence in $X_u$. As each subblock of $w$ is a limit of subblocks of $u$, it follows that~\eqref{eq:sv1} is still valid if we replace $u$ by $w$. But this in turn implies that any $w\in X_u$ is mean slowly varying.

Now, let $\nu$ be an ergodic shift-invariant measure on $X_u$, and
let $w\in X_u$ be $\nu$-generic ($\nu$-almost all sequences in
$X_u$ are $\nu$-generic). Then $w$ is mean slowly varying, and
by Proposition~\ref{p:slowlyvar}, the measure-theoretic system
$(X_u,\nu,S)$ is isomorphic to $(X_u,\nu,\Id)$. In particular,
its Kolmogorov entropy is zero.

By the variational principle (see e.g.\ \cite[Section 8.2]{walters},
the topological entropy of $(X_u,S)$ must be also zero.
\end{proof}

\begin{Remark}\label{r:slvar} As a matter of fact, the above
proof shows that whenever $u$ is a slowly varying function, any
$w\in X_u$ is also a slowly varying function. More
than that, we can make the following observation: by the definition of $X_u$, each $w\in X_u$
can be written as
\[
 w = \lim_{k\to\infty} S^{n_k}u,
\]
where $(n_k)$ is a non-decreasing sequence of non-negative integers. Either $(n_k)$ is bounded,
and then $w=S^nu$ for some $n\ge0$, or $n_k\to\infty$, and then as $u$ is slowly varying, $w$ must be
of the form $(w_1,w_1,w_1,\ldots)$ for some
$w_1\in\mathbb{D}$.\end{Remark}

In the following remark we provide some more observations on mean
slowly varying functions.

\begin{Remark}\label{r:moreMSV}(a)
Note that the subshift generated by a \emph{mean} slowly varying function
$u:\N\to\D$ can be of infinite entropy: we can modify a
slowly varying function on a subset of density zero to get a
mean slowly varying function generating the full subshift $(\D^\N,S)$.

(b) If $u$ is mean slowly varying and if, for each $\ell\geq1$, the limit
\[
\lim_{N\to\infty}\frac1{N}\sum_{1\leq n\leq N}u^{\ell}(n)
\]
exists then there is only one Furstenberg system of $u$: there exists a shift-invariant measure $\nu$ such that
$E_N(u)\tend{N}{\infty}\nu$.
Indeed, note that
for each $k\geq1$, $\ell_r\ge 1,j_r\ge 0$ for $r=1,\ldots,k$, we have
\[
 \left|\frac1N\sum_{1\leq n\leq N}u^{\ell_1}(n+j_1)\cdot\ldots\cdot u^{\ell_k}(n+j_k)-
 \frac1N\sum_{1\leq n\leq N}u^{\ell_1+\ldots+\ell_k}(n)\right|\tend{N}{\infty}0.
\]

(c) We can repeat word for word all arguments for the logarithmic
 averages.  In particular, if $u$ is mean slowly varying and, for each $\ell\geq1$, the limit
\beq\label{log26}
\lim_{N\to\infty}\frac1{\log N}\sum_{1\leq n\leq N}\frac1n u^{\ell}(n)
\eeq
exists, then there is only one logarithmic Furstenberg system (see Section~\ref{s:logChowla}):
there exists a shift-invariant measure $\nu$ such that
\[
\frac1{\log N}\sum_{1\leq n\leq N}\frac1n\delta_{S^n u}\tend{N}{\infty}\nu.
\]

(d) For $u(n)=n^{it}$, we have the logarithmic
assumption~\eqref{log26}  satisfied (in fact, each such limit is
zero whenever  $\ell\neq0$),
so there is only one logarithmic Furstenberg system,
as already noticed in \cite{Fr-Ho3} but we have uncountably
many Furstenberg systems (cf.\ Corollary~\ref{c:FSnit}).
\end{Remark}

As it follows from Proposition~\ref{p:wlasnosci}, a slowly
varying function $u:\N\to\D$ either has one Furstenberg system, or it has uncountably many
different Furstenberg systems. By different, we mean that
we obtain different measures. The measure-theoretic
dynamical systems given by these measures can however be all isomorphic
which is the case in the example below.

We consider now a mean slowly varying function $u$ such that $|u(n)|=1$ for all $n\geq1$.
Assume, moreover, that
\[
 \lim_{m\to\infty}\frac1{N_m}\sum_{0\leq n< N_m}\delta_{S^nu}=\lim_{m\to\infty}E_{N_m}(u)=\nu
\]
for an increasing sequence $(N_m)$.
It follows from Proposition~\ref{p:slowlyvar} that $\nu$ is supported on
the set $\{(z,z,\ldots):\:z\in\bs^1\}$.
Moreover, $Z_0^k(z,z,\ldots)=z^k$.
Hence, we can identify $\nu$ with a measure
$\kappa$ defined on $\bs^1$ such that
$$
\lim_{m\to\infty}\frac1{N_m}\sum_{1\leq n\leq N_m}u^k(n)=
\int Z_0^k\,d\nu=\int_{\bs^1}z^k\,d\kappa .$$
It follows that, varying $k\in\Z$, the LHS of the equation
above provides the Fourier transform of the measure
$\kappa$ which we are seeking.
In the particular case where $u(n)=n^{i}$,
we have
\[
 \frac1{N_m}\sum_{1\leq n\leq N_m} n^{ik}=\frac{N_m^{ik}}{1+ik}+{\rm o}(1),
\]
thus the Fourier coefficients of $\kappa$ are given by the limits of $\frac{N_m^{ki}}{1+ik}$ as $m\to\infty$.
The case $k=1$ yields that $N_m^i\to c$,
where $|c|=1$, and we obtain that
$$
\int_{\bs^1}z^k\,d\kappa = \frac{c^k}{1+ik},\;k\in\Z.$$
Consider, $\kappa':=\kappa\ast\delta_{\ov{c}}$. Then
$$
\int_{\bs^1}z^k\,d\kappa' =\frac1{1+ik},\;k\in\Z.$$
Since $(\frac1{1+ik})$ is an $\ell^2$- sequence,
the measure $\kappa'$ is absolutely continuous with respect to the (normalized) Lebesgue measure
$\Leb_{\bs^1}$,
with density $g$ equal to
$g(z)=\sum_{k=-\infty}^\infty\frac1{1+ik}z^k$.
In fact, noting that for all $k\in\Z$,
\[
 \int_0^1 e^{2\pi x}e^{i2\pi k x}dx = \frac{e^{2\pi}-1}{2\pi(1+ik)},
\]
we see that, for $x\in[0,1)$,
\[
g(e^{i2\pi x}) = \frac{2\pi e^{2\pi x}}{e^{2\pi}-1}.
\]
Finally, note that  the sequence $(n^i)$ is dense in $\bs^1$
as the sequence
$(\log n)$ is dense in $[0,2\pi)$ mod~$2\pi$.


\begin{Cor} \label{c:FSnit}The family of Furstenberg systems of
$u(n)=n^i$ consists of uncountably many different systems
given by all rotations of $g(z)dz$.
All of them are isomorphic to the identity on the circle
with Lebesgue measure
(and ergodic components are Dirac measures on the circle).
Moreover, $X_u=\{(z,z,\ldots):\: z\in\bs^1\}\cup\{S^nu:n\in\N\}$.
\end{Cor}

%

\subsection{MRT arithmetic functions do not satisfy Sarnak's conjecture}

As MRT arithmetic functions mostly behave like Archimedean
characters on very large intervals, we will use some
ideas presented in the preceding section to prove that
MRT arithmetic functions do not satisfy Sarnak's conjecture:
for any MRT function $\bfu$ there exists a zero entropy
topological system which outputs a sequence having some
correlation with $\bfu$.

Let $\bfu$ be an MRT arithmetic function, and let $(t_m)$ and
$(s_m)$ be as in Definition~\ref{def:MRT}. We set, for each
$m\in\N$,  $r_m:=\lfloor s_{m}^{3/2}\rfloor$. Then we define a
new arithmetic function $v=\bigl(v(n)\bigr)_{n\in\N}\in(\bs^1)^\N$
by setting for each $m$:
\[
 v(n) := \begin{cases}
            1 &\text{ for } t_m < n \le r_{m+1},\\
           n^{is_{m+1}} &\text{ for }
           r_{m+1} < n \le t_{m+1}.
        \end{cases}
\]
Then, for each $m$, we have
\[
 \left| v(n+1) - v(n) \right| =
 \begin{cases}
  0 &\text{ if }t_m < n < r_{m+1},\\
  {\rm O}\left(s_{m+1}^{-1/2}\right)&
  \text{ if } r_{m+1} < n < t_{m+1}.
 \end{cases}
\]
Even though $v$ may not be slowly varying because of the possible
jumps in $t_m$ and in $r_{m+1}$,
the property described in Equation~\eqref{eq:sv1} is still valid,
as these jumps are bounded by 2 and are separated by gaps whose lengths
tend to $\infty$. Therefore,
the proof of Proposition~\ref{prop:slowly} also applies to $v$,
and we have $h_{\text{top}}(X_v,S)=0$.

But, since $\frac{r_{m+1}}{t_{m+1}}\tend{m}{\infty}0$,
in view of
Proposition~\ref{prop:MRT}, we have
\[
 \frac{1}{t_{m+1}} \sum_{1\le n\le t_{m+1}} \bfu(n) \overline{v(n)} \tend{m}{\infty} 1.
\]
We thus have found a topological dynamical system $(X_v,S)$ of zero topological entropy, a point $v\in X_v$ and a continuous map $f:X_v\to\mathbb{C}$
(the conjugate of the zero-coordinate map), such that the sequence $\bigl(f(S^nv)\bigr)$ is not orthogonal to $\bfu$, in the sense that
\[
 \frac{1}{N} \sum_{1\le n\le N} \bfu(n) f(S^n v) \longrightarrow\hspace{-13pt}/\hspace{13pt} 0 \text{ as }N\to\infty.
\]

\section{Further properties of MRT functions} \label{s:rozne}

\subsection{MRT functions satisfy logarithmic
Chowla conjecture along a subsequence} \label{s:logChowla}

The purpose of this section is to study what the Main Theorem becomes if we consider logarithmic averages instead of
usual averages. A \emph{logarithmic Furstenberg system} of an arithmetic function $u:\N\to\D$ is defined as in Section~\ref{s:Furstenberg}
as a measure-theoretic dynamical system $(X_u,\nu,S)$, where $\nu$ is now a weak$^\ast$ limit of a subsequence of the
\emph{logarithmic} empirical measures
\[
   E_{N}^{\log}(u):=\frac{1}{L_{N}} \sum_{1\le n\le N}\frac{1}{n} \delta_{S^{n-1}u},
\]
with  $L_N:=1+\frac{1}{2}+\cdots+\frac{1}{N}$ ($N\ge1$).

We recall the classical relation between the logarithmic and usual averages, obtained by summation by parts:
\begin{equation}
 \label{eq:sum_by_parts}
 E_{N}^{\log}(u) = \frac{1}{L_{N}} \sum_{1\le n\le N-1}\frac{1}{n+1} E_n(u) + \frac{1}{L_{N}} E_{N}(u).
\end{equation}

We recall that the weak$^\ast$-topology turns
$M((\bs^1)^{\N})$ into a  compact metrizable space, and that a possible metric is given by
\[
 \Delta(\nu,\kappa)=\sum_{j\geq 1}\frac1{2^j}
\left|\int f_j\,d\nu-\int f_j\,d\kappa\right|,
\]
where $(f_j)_{j\geq1}$ is a countable family generating a dense subspace of $C((\bs^1)^{\N})$.
We can for example take for $(f_j)_{j\geq1}$
 the family of
$\{Z_{i_1}^{r_1}\cdot\ldots\cdot Z_{i_t}^{r_t}:\:
t\geq 1,i_k\ge0,r_k\in\Z\}$
taken in any order. Then, the metric is globally bounded by~2.
Note also that for convex combinations, we have
\beq\label{log1}
\Delta\left(\int\nu_\gamma\,dP(\gamma),\int \kappa_\gamma\,dP(\gamma)\right)\leq
\int \Delta(\nu_\gamma,\kappa_\gamma)\,dP(\gamma).
\eeq
We need the following simple observation.

\begin{Lemma}\label{l:log1}
Assume that $\mu_k\to\mu$ and let, for $d\ge1$, $\kappa_d:=\sum_{k=d}^{d+a_d}
\alpha_k^{(d)}\mu_k+\beta^{(d)}\rho_d$, where $\alpha^{(d)}_k, \beta^{(d)}\geq 0$,
$\sum_{k=d}^{d+a_d}\alpha_k^{(d)}=1-\beta^{(d)}$. If $\beta^{(d)}\to0$ then
$\kappa_d\to\mu$.\end{Lemma}
\begin{proof} This follows immediately from~\eqref{log1}.\end{proof}

Now, we consider an MRT arithmetic function $\bfu$, and we take $(t_m)$ and $(s_m)$ as in Definition~\ref{def:MRT}.
The proof of  Main Theorem (see Section~\ref{s:dowodMRT})
yields the following:

\begin{Lemma}\label{l:log2}
Fix $d\geq0$ and choose
\[
 \begin{cases}
  \frac1{d+1}<\beta_d<\beta'_d<\frac1d&\text{ if }d\ge1,\\
  1<\beta_0<\beta'_0<2&\text{ if }d=0.
 \end{cases}
\]
Let $\vep>0$. Then, for each $m$ large enough,
\[
\Delta\left(
E_N(\bfu),\nu_d\right)<\vep
\]
uniformly in $s_{m+1}^{\beta_d}\leq N\leq s_{m+1}^{\beta'_d}$.
\end{Lemma}
\begin{proof}
Assume the result does not hold. Then for infinitely many integers $m$ we can find $s_{m+1}^{\beta_d}\leq N_m\leq s_{m+1}^{\beta'_d}$
such that
\[
\Delta\left(
E_{N_m}(\bfu),\nu_d\right)\ge \vep.
\]
On the other hand, the proof of Proposition~\ref{prop:final} shows that, along such a sequence $(N_m)$ we must have
\[
E_{N_m}(\bfu) \to \nu_d,
\]
which yields a contradiction.
\end{proof}

Now, let us fix integers $1\le D_1<D_2$. We consider the convergence of logarithmic empirical measures of $\bfu$ along the increasing sequence of integers
$(N_m):=\left( \left\lfloor  s_{m+1}^{1/D_1}  \right\rfloor\right)$. Given a small real number $\varepsilon>0$, set for each $D_1\le d\le D_2$
\[
   \beta_d:= \frac{1}{d+1}+ \frac{\varepsilon}{2d(d+1)} \quad \text{ and } \quad \beta'_d:= \frac{1}{d} - \frac{\varepsilon}{2d(d+1)}.
\]
Using~\eqref{eq:sum_by_parts}, we write $E_{N_m}^{\log}(\bfu)$ as a convex combination of the empirical measures $E_n(\bfu)$, $1\le n\le N$.
We partition $\{1,\ldots,N_m\}$ as $\bigsqcup_{D_1\le d\le D_2} I_d^m\sqcup J^m$, where
\[
 I_d^m := \left\{n\in\N: s_{m+1}^{\beta_d} \le n\le s_{m+1}^{\beta'_d}\right\},
\]
and
\[
   J^m := \{1,\ldots,N_m\} \setminus \bigsqcup_{D_1\le d\le D_2} I_d^m.
\]
By Lemma~\ref{l:log2}, for $m$ large, $E_n(\bfu)$ is close to $\nu_d$ for $n\in I_d^m$. The weight of $I_d^m$ in the convex combination is
\[
   \frac{1}{L_{N_m}} \sum_{n\in I_d^m} \frac{1}{n+1} \equival{m}{\infty} D_1 (\beta'_d-\beta_d) = D_1 \left( \frac{1}{d} - \frac{1}{d+1} \right)(1-\varepsilon).
\]
It follows that the total weight of $\bigsqcup_{D_1\le d\le D_2} I_d^m$ is asymptotic, as $m\to\infty$, to
\[
   D_1 \left( \frac{1}{D_1} - \frac{1}{D_2} \right)(1-\varepsilon) > 1-\varepsilon - \frac{D_1}{D_2},
\]
and then for $m$ large enough the weight of $J^m$ is bounded by $\varepsilon+\frac{D_1}{D_2}$.
In view of Lemma~\ref{l:log2} and~\eqref{log1}, any weak$^\ast$ limit of $E_{N_m}^{\log}(\bfu)$ can be written as
\[
 D_1(1-\varepsilon) \sum_{D_1\le d \le D_2} \left( \frac{1}{d} - \frac{1}{d+1} \right) \nu_d + \alpha\rho,
\]
where $\rho$ is some shift-invariant probability measure on $X_u$ and $0\leq\alpha\le \varepsilon+\frac{D_1}{D_2}$.
Letting $\varepsilon\to 0$ and $D_2\to\infty$, we see (cf. Lemmma~\ref{l:log1}) that for any $D_1\ge1$ there exists a logarithmic Furstenberg system of $\bfu$ whose invariant measure is
\[
   D_1 \sum_{d\ge D_1} \left( \frac{1}{d} - \frac{1}{d+1} \right) \nu_d.
\]
Note that, under this measure, the distribution of $\left(Z_0,\ldots,Z_{D_1}\right)$ is
$\left(\Leb_{\bs^1}\right)^{\otimes (D_1+1)}$.
By Proposition~\ref{c:rou10} (which is also valid in the logarithmic case), we can also find $\left(\Leb_{\bs^1}\right)^{\otimes \N}$ as an invariant measure of a logarithmic Furstenberg system of $\bfu$.

Thus, we have proved the following result:

\begin{Cor} Each $\bfu\in \MRT$ satisfies the logarithmic Chowla
conjecture along a subsequence.\end{Cor}

\subsubsection*{Getting $\nu_0$}
Under the assumptions given in Definition~\ref{def:MRT}, we can easily modify the above argument to incorporate $\nu_0$ in the weak$^\ast$ limit of $E_{N_m}^{\log}(\bfu)$: we just have to take $D_1=0$ and choose $\beta_0:=1+\varepsilon$, $\beta'_0:=2-\varepsilon$. The weak$^\ast$ limit we get has then the form
\[
 \frac{1}{2} \left(\nu_0 + \sum_{d\ge1} \left(\frac{1}{d} - \frac{1}{d+1}\right) \nu_d \right).
\]
If in the construction of the MRT sequence we add the extra assumption that
\[
    s_{m+1}^{a_m} \le t_{m+1},\quad\text{ with }a_m\to\infty,
\]
which is compatible with the rest, then we can now take $\beta'_0$ as large as we want, so that the weight of $\nu_0$ in the weak$^\ast$ limit is as close to 1 as we want. Finally, we get a logarithmic Furstenberg system of $\bfu$ with $\nu_0$ as an invariant measure.

It is not clear however if we can get $\nu_d$ ($d\ge1$) for some logarithmic Furstenberg system of an MRT function.

%
%

\subsection{Absence of zero mean on typical short interval}
Motivated by Matom\"aki-Radziwi\l\l's theorem \cite{Ma-Ra} concerning strongly aperiodic multiplicative functions, we say that $\bfu:\N\to\D$
has {\em zero mean on typical short interval} if
$$
\lim_{M,H\to\infty, H={\rm o}(M)}\frac1M \sum_{1\leq m\leq M}\left|
\frac1H\sum_{0\leq h<H}u(m+h)\right|=0.$$

\begin{Prop}\label{p:si}
Assume that $\liminf_{\ell\to\infty}
\frac1\ell\sum_{1\leq j\leq \ell}|u(j)|
=:\alpha>0$.
If $u$ has identity as a Furstenberg system then $u$ has no
zero mean on short intervals.\end{Prop}
\begin{proof} Suppose that $u$ has zero mean on short intervals.
Assume that
$$\lim_{k\to\infty}E_{M_k}(u)=\kappa.$$ Now, the system $(S,X_u,\kappa)$ is the identity if
and only if $Z_0=Z_0\circ S$ $\kappa$-a.e.

Let $\vep>0$ and choose $H_0$ so that for $H>H_0$, we have
$$
\limsup_{k\to\infty}\frac1{M_k} \sum_{1\leq m\leq M_k}\left|
\frac1H\sum_{0\leq h<H}u(m+h)\right|<\vep.$$
Take $H>H_0$. Then
$$
\limsup_{k\to\infty}\frac1{M_k} \sum_{1\leq m\leq M_k}\left|
\frac1H\sum_{0\leq h<H} u(m+h)\right|=$$
$$
\limsup_{k\to\infty}\frac1{M_k} \sum_{1\leq m\leq M_k}\left|
\frac1H\sum_{0\leq h<H} Z_0(S^{h+m}u)\right|=
$$
$$
\limsup_{k\to\infty}\int_{X_u}\left|\frac1H\sum_{0\leq h<H}
Z_0\circ S^h\right|\,d E_{M_k}(u)=
$$ $$
\int_{X_u}\left|\frac1H\sum_{0\leq h<H}
Z_0\circ S^h\right|\,d\kappa=\int_{X_u}|Z_0|\,d\kappa\geq\alpha>0,$$
a contradiction.
\end{proof}

\begin{Cor} If $\bfu\in\MRT$ then $\bfu$ has no zero mean on short intervals.
\end{Cor}

\subsection{Strong stationarity}
We now show that, for each $d\geq0$, the stationary process
$(Z_n)$ with distribution $\nu_d$ is {\em strongly stationary} (see \cite{Je}), i.e.
for each $s_1<\ldots<s_k$ and each $r\geq1$, the distributions
of the vectors $(Z_{s_1},\ldots,Z_{s_k})$ and
$(Z_{rs_1},\ldots,Z_{rs_k})$ are the same.

Fix $d\geq0$ and we recall that the process $(Z_n)$ with the
distribution $\nu_d$ is the same as of the process $(f\circ T_d^n)$,
see~\eqref{tor2} in~Appendix for the definition of $T_d$ and $f(x_1,\ldots,x_d)=e^{i2\pi  x_d}$. Then
$$
f\circ T_d^s(x_1,\ldots,x_d)=e^{2\pi i(
{s\choose d-1}x_1+{s\choose d-2}x_2+\ldots+{s\choose 0}x_d)}.
$$
 It follows that for each $s_1<\ldots<s_k$ and each choice of
integers $j_1,\ldots,j_k$, we have
$$\mathbb{E}
Z_{s_1}^{j_1}\cdot\ldots\cdot Z_{s_k}^{j_k}\neq 0 \Leftrightarrow$$
$$
\sum_{i=1}^k j_i{s_i\choose\ell}=0\text{ for }\ell=0,\ldots,d-1
\Leftrightarrow
\sum_{i=1}^ks_i^{\ell}=0\text{ for }\ell=0,\ldots,d-1.$$
Therefore,
$$
\mathbb{E}
Z_{s_1}^{j_1}\cdot\ldots\cdot Z_{s_k}^{j_k}=0 \Leftrightarrow
\mathbb{E}
Z_{rs_1}^{j_1}\cdot\ldots\cdot Z_{rs_k}^{j_k}=0 \text{ for each }r\geq1.
$$
Moreover,
$\mathbb{E}
Z_{s_1}^{j_1}\cdot\ldots\cdot Z_{s_k}^{j_k}\neq0$ is equivalent to
$\mathbb{E}
Z_{s_1}^{j_1}\cdot\ldots\cdot Z_{s_k}^{j_k}=1$. It follows that
the process $(Z_n)$ with the distribution $\nu_d$ is strongly
stationary.

It follows that all Furstenberg systems which have been found in
the paper are given by stationary processes
which are strongly stationary. Note also that since every convex
combination of strongly stationary processes remains strongly
stationary, also the logarithmic Furstenberg systems determined in
Section~\ref{s:logChowla} are strongly stationary
which fits perfectly to a general result
of Frantzikinakis and Host \cite{Fr-Ho2}, \cite{Fr-Ho3} about logarithmic Furstenberg systems of
strongly aperiodic multiplicative functions.

\medskip

We end by asking the following question, which has been suggested to us by both Nikos Frantzikinakis and Florian Richter: can we find a Furstenberg system of some $\bfu\in \MRT$ which is the direct product of a Bernoulli shift and a unipotent system?

\section{Appendix}\label{s:quasi-discrete}
We aim at showing that from the dynamical point of view there is a close relation between the processes $(Y^{(d)})_{d\geq0}$ which appeared in Section~\ref{s:proof} and the concept of quasi-eigenfunction in ergodic theory.

\subsection{Algebraic constraints  and quasi-eigenfunctions}
 Given a  standard Borel probability space $\xbm$, let $\mathcal{M}(X)$ denote the set of all measurable functions of modulus~1 defined on $X$. Endowed with the pointwise multiplication and the $L^2$-topology, it becomes a Polish group.
Given an automorphism $T$ on $\xbm$, set
$$E_0(T):=\{g\in \mathcal{M}(X):\: g\circ T=g\}$$
and then inductively, define
$$
E_{d+1}(T):=\{f\in\mathcal{M}(X) :\: f\circ T/f\in E_d(T)\},\;d\geq0.$$
Note that each $E_d(T)$ is a group which is called the group of $d$-quasi-eigenfunctions. For $w\in \mathcal{M}(X)$ and $m\geq1$, we set $$w^{(m)}=w\circ T^{m-1}\cdot\ldots\cdot w$$ with $w^{(0)}=1$. Assume that $g\in E_d(T)$ and let $g\circ T=hg$ with $h=g\circ T/g\in E_{d-1}(T)$. Then, for each $\ell_1,\ldots,\ell_k$, we have
$$
\prod_{i=1}^kg\circ T^{\ell_i}=\prod_{i=1}^k
h^{(\ell_i)}g =g^k\prod_{i=1}^kh^{(\ell_i)}=$$$$
g^k\prod_{i=1}^k\left(\prod_{j=0}^{\ell_i-1}h\circ T^j\right)=
g^k\prod_{i=1}^k\prod_{j=0}^{\ell_i-1}\left(\frac{h\circ  T}{h}\right)^{(j)}h=$$$$
g^k\prod_{i=1}^k h^{\ell_i}\prod_{j=0}^{\ell_i-1}\left(\frac{h\circ T}{h}\right)^{(j)}=$$$$g^kh^{\sum_{i=1}^k\ell_i}
\prod_{i=1}^k\prod_{j=0}^{\ell_i-1}\left(\frac{h\circ T}{h}\right)\circ T^{j-1}\cdot\ldots\cdot \left(\frac{h\circ T}{h}\right)=$$$$
g^kh^{\sum_{i=1}^k\ell_i}
\prod_{i=1}^k\prod_{j=0}^{\ell_i-1}\prod_{p=0}^{j-1}\left(\frac{(h\circ T/h)\circ T}{h\circ T/h}\right)^{(p)}\frac{h\circ T}{h}=$$
$$
g^kh^{\sum_{i=1}^k\ell_i}
\prod_{i=1}^k\prod_{j=0}^{\ell_i-1}\left(\frac{h\circ T}{h}\right)^j\prod_{p=0}^{j-1}\left(\frac{(h\circ T/h)\circ T}{h\circ T/h}\right)^{(p)}=$$$$
g^kh^{\sum_{i=1}^k\ell_i}\left(\frac{h\circ T}{h}\right)^{\sum_{i}^k\frac{\ell_i(\ell_i-1)}{2}}
\prod_{i=1}^k\prod_{j=0}^{\ell_i-1}\prod_{p=0}^{j-1}\left(\frac{(h\circ T/h)\circ T}{h\circ T/h}\right)^{(p)}=\ldots$$
It follows that for the stationary process $(g\circ T^n)_{n\in\N}$ the following holds: whenever $\ell_i,\ell'_i$ satisfy $\sum_{i=1}^k\ell^j_i=\sum_{i=1}^k\ell^{\prime j}_i$ for $j=0,1,\ldots,d$, we have
\beq\label{tor1}
\prod_{i=1}^kg\circ T^{\ell_i}=\prod_{i=1}^k g\circ T^{\ell'_i}\eeq
provided that $g\in E_d(T)$.

In fact, the processes given by quasi-eigenfunctions are the only satisfying the algebraic relation~\eqref{tor1}.

\begin{Prop}\label{p:rou1} Assume that $(Z_n)$ is a stationary $\bs^1$-valued process. Then for each $d\geq0$, $Z_0\in E_d(T)$ if and only if
\beq\label{rou30a}
Z_{\ell_1}\cdot\ldots\cdot Z_{\ell_k}=
Z_{\ell'_1}\cdot\ldots\cdot Z_{\ell'_k}
\eeq
for each $\ell_i,\ell'_i$ for which $\sum_{i=1}^k\ell^j_i=\sum_{i=1}^k\ell^{\prime j}_i$ for $j=0,1,\ldots,d$ assuming that \beq\label{rou30}
\{\ell_1,\ldots,\ell_l\}\cap\{\ell'_1,\ldots,\ell'_k\}=\emptyset.\eeq
\end{Prop}
\begin{proof}
We can assume that $Z_n=Z_0\circ T^n$ for an automorphism $T$ of $\xbm$.  Now, we proceed by induction. Since the first condition ($d=0$) is just
$Z_0\circ T^{\ell_1}\cdot \ldots\cdot \circ Z_0\circ T^{\ell_k}=Z_0\circ T^{\ell'_1}\cdot\ldots\cdot Z_0\circ T^{\ell'_k}$ whichever  numbers we take, we have (by taking $k=1$, $\ell_1=1$ and $\ell'_1=0$) $Z_0\circ T=Z_0$, so $Z_0\in E_0(T)$.

Notice that condition~\eqref{rou30} is superfluous since if  some of the numbers $\ell_i$ are equal to $\ell'_j$, then we can just cancel out the relevant factors.
So, assume now that $Z_0$ satisfies~\eqref{rou30a} and let  $f=Z_0\circ T/Z_0$. We need to prove that $f$ satisfies the relation~\eqref{rou30a} up to $d$ (if so, then by the induction assumption $f\in E_d(T)$, whence $Z_0\in E_{d+1}(T)$). That is, we want to prove that
$$
f\circ T^{\ell_1}\cdot\ldots\cdot f\circ T^{\ell_k}=
f\circ T^{\ell'_1}\cdot\ldots\cdot f\circ T^{\ell'_k}$$
provided that $\sum_{i=1}^k\ell_i^j=\sum_{i=1}^k\ell_i^{\prime j}$ for $j=0,1,\ldots, d$. Equivalently, we want to show that
$$
Z_0\circ T^{\ell_1+1}\cdot \ldots\cdot Z_0\circ T^{\ell_k+1}\cdot
Z_0\circ T^{\ell'_1}\cdot\ldots\cdot Z_0\circ T^{\ell'_k}=$$
$$
Z_0\circ T^{\ell'_1+1}\cdot \ldots\cdot Z_0\circ T^{\ell'_k+1}
\cdot Z_0\circ T^{\ell_1}\cdot\ldots\cdot Z_0\circ T^{\ell_k}.$$
But we clearly have
$$
(\ell_1+1)^{d+1}+\ldots+(\ell_k+1)^{d+1}+\ell_1^{\prime d+1}+\ldots+ \ell_k^{\prime d+1}=$$$$
(\ell'_1+1)^{d+1}+\ldots+(\ell'_k+1)^{d+1}+\ell_1^{ d+1}+\ldots+ \ell_k^{d+1}$$
(let alone if we replace $d+1$ by a smaller $j$), hence the result.
\end{proof}

\begin{Remark} Assume additionally that $T$ is totally ergodic ({\em i.e.} all non-zero powers are ergodic). Then, following \cite{Ab}, $T$ has {\em quasi-discrete spectrum} if $\ov{\rm span}\left(\bigcup_{d\geq0}E_d(T)\right)=L^2\xbm$  (sometimes, $T$ is called an Abramov automorphism). Quasi-discrete spectrum automorphisms are  basically affine automorphisms of compact, Abelian, metric groups, see \cite{Ha-Pa} for more details.\end{Remark}

\begin{Remark} Knowing that $Z_0\in E_d(T)$ does not determine the whole process $(Z_m)_{m\in\N}$. Indeed, for example, if we take the ergodic decomposition of $T$, then the process $(Z_m)$ (considered with respect to an ergodic component) will still satisfy the same algebraic relation even though the distribution of the process may have changed.\end{Remark}

\subsection{Algebraic constraints and $n$-independence}
We will be interested in the (very) non-ergodic case.
Given $d\geq1$, consider the (unipotent) automorphism $T_d:\T^d\to \T^d$, defined by:
\beq\label{tor2}
T_d(x_1,\ldots, x_d):=(x_1,x_1+x_2,\ldots, x_{d-1}+x_d)\eeq
(hence $T_1$ is just the identity on $\T$). Set $f(x_1,\ldots,x_d)=e^{i2\pi  x_d}$. Then, $\big(f\circ T/f\big)(x_1,\ldots,x_d)=e^{i2\pi  x_{d-1}}$, so by induction, it is easy to show that $f\in E_d(T_d)$.

Denote by ${\rm Erg}^\perp$ the class of automorphisms disjoint from all ergodic automorphisms.
\begin{Prop}\label{p:tor1}
For each $d\geq1$, $T_d\in {\rm Erg}^{\perp}$. Moreover, the stationary process $(f\circ T_d^m)_{m\in\N}$ is $d$-independent, i.e. for all $m_1<\ldots<m_d$ the variables $f\circ T_d^{m_1},\ldots, f\circ T_d^{m_d}$ are independent.\end{Prop}
\begin{proof}
To prove that $T_d\in{\rm Erg}^\perp$, we first notice that its ergodic decomposition is given by the  tori (on which we consider the relevant Lebesgue measure) $\T^{d-1}_x=\T^{d-1}$ with $x\in\T$ irrational (and we consider Lebesgue measure\footnote{That is, we identify the space of ergodic components with $(\T,\Leb_{\T})$.} on $\T$) on which the action
$T_{d,x}$ is given by
$$
(x_2,x_3\ldots,x_d)\mapsto (x_2+x,x_2+x_3,\ldots, x_{d-1}+x_d).$$
If an ergodic automorphism $S$ acting on $\ycn$ is non-disjoint with $T_{d,x}$ then since $T_{d,x}$ is an ergodic compact group extension of an irrational rotation,  the associated Koopman operator $U_S$ must share a common eigenvalue with the Koopman operator $U_{T_{d,x}}$ given by $T_{d,x}$ (cf.\ \cite{Gl}, Chapter~6). As $U_S$ can have only countably many eigenvalues and the measure on the space of ergodic components is continuous, we can assume that $S$ is disjoint with all ergodic components. Take any joining $\rho$ of $T_d$ and $S$. Let
$$
\rho=\int\rho_\gamma\,dP(\gamma)$$
be its ergodic decomposition. Then, for $P$-a.e.~$\gamma$, the projection $\rho_\gamma|_Y$ of $\rho_\gamma$ on $Y$ equals $\nu$ since $\nu$ is ergodic. Moreover,
\beq\label{tor4}
\Leb_{\T^{\ot d}}=\rho|_{\T^d}=\int \rho_\gamma|_{\T^d}\,dP(\gamma)
\eeq
is a decomposition of $\Leb_{\T^{\ot d}}$ into ergodic measures. By the uniqueness of ergodic decomposition, it is \eqref{tor4} which is the ergodic decomposition, and (by disjointness) we obtain that
$$
\rho=\int_{\T}(\delta_x\ot\left(\Leb_{\T}\right)^{\ot (d-1)})\ot \nu\,d\Leb_{\T}(x)$$
and the first claim easily follows.

It is not hard to check that by elementary properties of Pascal triangle:
\beq\label{tor3a}
T_d^r(x_1,\ldots, x_d)=\Big(x_1,\ldots, {r\choose d-1}x_1+{r\choose d-2}x_2+\ldots+{r\choose 0}x_d\Big).\eeq
Now, choose any integers $r_1<r_2<\ldots r_d$ and $q_j\in\Z$ for $j=1,\ldots,d$. We want to show that the distribution of the vector $(f^{q_1}\circ T_d^{r_1}, \ldots, f^{q_d}\circ T_d^{r_d})$ is $\left(\Leb_{\T}\right)^{\otimes d}$ and for that we need to check that
\beq\label{tor3}\mathbb{E}(f^{q_1}\circ T_d^{r_1}\cdot \ldots\cdot f^{q_d}\circ T_d^{r_d})=0\eeq
unless $q_1=\ldots=q_d=0$. Now, in view of~\eqref{tor3a}, the negation of \eqref{tor3} is equivalent to
$$
\sum_{j=1}^dq_j{ r_j\choose k}=0\text{ for }k=0,1,\ldots,d-1.$$
It is not hard to see that this system of linear equations is equivalent to
$$
\sum_{j=1}^d q_jr_j^k=0\text{ for }k=0,1,\ldots,d-1.$$
However, the determinant here is the Vandermonde determinant and since the numbers $r_j$ are pairwise different, our claim follows.
\end{proof}

\begin{Remark} By the proof of Proposition~\ref{prop:nud} (see also Remark~\ref{rem:nud}), under the assumptions of Proposition~\ref{prop:nud}, we obtain that the dynamical system corresponding to the stationary process $(Z_n)$ is (up to isomorphism) just $T_d$.
\end{Remark}

\begin{Remark}
If $T$ acting on $\xbm$ is additionally ergodic then $E_0(T)$ consists of the constants while $E_1(T)$ consists of the eigenfunctions of the Koopman operator $U_T$ acting on $L^2\xbm$. Hence, if $(Z_n)_{n\in\N}$ is a stationary process satisfying \beq\label{tor10}Z_0\cdot Z_2=Z_1^2\eeq then, using also Proposition~\ref{p:rou1}, the ergodic components of the corresponding dynamical system must have discrete spectra.\footnote{Note that by Proposition~\ref{p:rou1} it follows that $Z_1=cZ_0$, whence \eqref{tor2} is satisfied whenever $\sum_{i=1}^k\ell_i=\sum_{i=1}^k\ell'_i$.} A prominent example of such a situation is the automorphisms $T_2(x_1,x_2)=(x_1,x_1+x_2)$ on $\T^2$ whose ergodic components are (all) irrational rotations (and $Z_0$ is given by $f_2:(x_1,x_2)\mapsto e^{i2\pi x_2}$).
\end{Remark}

\vspace{2ex}

\noindent
{\bf Acknowledgments}
Research of the first and second authors supported by Narodowe Centrum Nauki grant UMO-2019/33/B/ST1/00364.
The second and third authors would like to thank the American Institute of Mathematics
for hosting a workshop {\em Sarnak's Conjecture}, where this work was first discussed.

\normalsize

\vspace{2ex}
\noindent
Faculty of Mathematics and Computer Science\\
Nicolaus Copernicus University, Toru\'n, Poland\\
gomilko@mat.umk.pl, mlem@mat.umk.pl\\
Laboratoire de Math\'ematiques Rapha\"el Salem, Universit\'e de Rouen Normandie\\ CNRS -- Avenue de l'Universit\'e -- 76801
Saint \'Etienne du Rouvray, France\\
Thierry.de-la-Rue@univ-rouen.fr


\begin{thebibliography}{99}
\bibitem{Ab-Le-Ru} H. El Abdalaoui. M. Lema\'nczyk, T. de la Rue,
{\em A dynamical point of view on the set of $\mathcal{B}$-free
integers},  Int.\ Math. Res.\ Not.\ IMRN 2015, no 16,
7258-7286.
\bibitem{Ab} L.M. Abramov, {\em Metric automorphisms with quasi-discrete spectrum}, Izv.
Akad. Nauk U.S.S.R., 26 (1962), 513-530.


\bibitem{Workshop} American Institute of Mathematics, workshop {\em Sarnak’s Conjecture}, December 2018,
http://aimpl.org/sarnakconjecture/3/


\bibitem{Be-Ku-Le-Ri}V. Bergelson, J.\ Ku\l aga-Przymus, M.\ Lema\'nczyk, F.\ Richter, {\em Rationally almost periodic   sequences, polynomial multiple recurrence and symbolic dynamics}, Ergodic Theory Dynam. Systems
    {\bf 39} (2019),  2332-2383.



\bibitem{Ch} S. Chowla, The Riemann Hypothesis and Hilbert’s Tenth Problem. Mathematics and Its Applications
4, Gordon and Breach Science Publishers, New York, 1965.


\bibitem{Denker}M. Denker, C. Grillenberger, and K. Sigmund, {\em Ergodic theory on compact
spaces}, Lecture Notes in Mathematics, Vol. 527, Springer-Verlag, Berlin-New York, 1976.



\bibitem{El1} P. Elliott, {\em  Multiplicative functions $|g| \leq 1$ and their convolutions: An overview},  S\'eminaire de Th\'eorie
des Nombres, Paris 1987-88. Progress in Mathematics 81 (1990), 63–75.
\bibitem{El2} P.D.T.A. Elliott, {\em On the correlation of multiplicative functions}, Notas Soc. Mat. Chile, 11(1):1-
11, 1992.

\bibitem{El3} P. Elliott, {\em On the correlation of multiplicative
and the sum of additive arithmetic functions}, Mem.
Amer. Math. Soc. 112 (1994), no. 538, viii+88pp.




\bibitem{Fl}L. Flaminio, {\em Mixing k-fold independent processes of zero entropy}, Proc. Amer. Math. Soc. 118 (1993), no. 4, 1263–1269.
\bibitem{Fr0}  N. Frantzikinakis, {\em Ergodicity of the Liouville system implies the Chowla conjecture}, Discrete Anal. 2017, Paper No. 19, 41 pp.
\bibitem{Fr}N. Frantzikinakis, {\em An averaged Chowla and Elliott conjecture along independent polynomials}, Int.
Math. Res. Not. IMRN 2018, no 12, 3721–3743.

\bibitem{Fr-Ho} N. Frantzikinkis, B. Host,
{\em Asymptotics for multilinear averages of multiplicative functions},
Math.\ Proc.\ Camb.\ Phil.\ Soc.\ 161 (2016), 87–101.

\bibitem{Fr-Ho2}N. Frantzikinakis, B. Host, {\em The logarithmic Sarnak conjecture for ergodic weights}, Ann. of Math. (2) 187 (2018), no. 3, 869–931.
\bibitem{Fr-Ho3}N. Frantzikinakis, B. Host,
{\em Furstenberg systems of bounded multiplicative
functions and applications},
Int. Math.\ Res.\ Not. IMRN 2021, no 8, 6077-6107.

\bibitem{Fu}H. Furstenberg, {\em Strict ergodicity and transformation of the torus}, Amer. J. Math. 83 (1961), 573–601.
    
\bibitem{Gl} E. Glasner, {\em Ergodic Theory via Joinings}, Mathematical Surveys and Monographs, vol. 101, American Mathematical Society, 2003.

\bibitem{Ha-Pa}F. Hahn, W. Parry, {\em Minimal dynamical systems with quasi-discrete spectrum},
J. London Math. Soc. 40 (1965), 309 - 323.

\bibitem{Je}E. Jenvey, {\em Strong stationarity and de Finetti’s theorem}, J. Anal. Math. 73 (1997), 1–18.
\bibitem{klurman2017}O. Klurman, {\em Correlations of multiplicative functions and applications}, Compositio Math. 153 (2017), 1622–1657.
\bibitem{Ma} L. Matthiesen, {\em Linear correlations of multiplicative functions}, Proc.\ LOndon Math. Soc.\ (3) 121 (2020), 372-425.
\bibitem{Ma-Ra}K. Matom\"aki, M. Radziwi\l\l, {\em Multiplicative functions in short intervals}, Annals of Math.\ 183
(2016), 1015–1056.
\bibitem{Ma-Ra-Ta} K. Matom\"aki, M. Radziwi\l\l, T. Tao,
{\em  An averaged form of Chowla’s conjecture},  Algebra Number Theory 9 (2015), no. 9, 2167–2196.
\bibitem{Ri}J. Rivat, {\em Analytic Number Theory},  in:
 Ergodic theory and dynamical systems in their interactions with arithmetics and combinatorics,  Lecture Notes in Math., {\bf 2213}, Springer, Cham, 2018, pp. 1-99.
\bibitem{Sa} P. Sarnak, {\em Three lectures on the M\"obius function, randomness and dynamics},
http://publications.ias.edu/sarnak/.
\bibitem{Ta} T. Tao, {\em The logarithmically averaged Chowla and Elliott conjectures for two-point correlations},
Forum Math. Pi 4 (2016), e8, 36 pp.
\bibitem{Ta-Te0} T. Tao, J.\ Ter\"av\"ainen, {\em Odd order cases of the logarithmically averaged Chowla conjecture}, J. Théor. Nombres Bordeaux 30 (2018), no. 3, 997–1015.

\bibitem{Ta-Te1} T. Tao, J.\ Ter\"av\"ainen, {\em The structure of logarithmically averaged correlations of multiplicative functions, with applications to the Chowla and Elliott conjectures}, Duke Math. J. 168 (2019), no. 11, 1977–2027.
\bibitem{Ta-Te}T. Tao. J. Ter\"av\"ainen, {\em
The structure of correlations of multiplicative functions at almost all scales, with applications to the Chowla and Elliott conjectures}, Algebra Number Theory 13 (2019), no. 9, 2103–2150.



\bibitem{walters}P. Walters, {\em An introduction to ergodic theory}, Graduate Texts in Mathematics, Vol. 79, Springer-Verlag, New York-Heidelberg-Berlin, 1982.
\end{thebibliography}
\end{document}